\documentclass[11pt]{article}
\usepackage{styles}
\usepackage{fullpage}
\newcommand{\Unif}{\mathrm{Unif}}
\newcommand{\Supp}{\mathrm{Supp}}
\newcommand{\Res}{\mathrm{Res}}

\title{Uncertainty Principles for the Number Theoretic Transform}
\author{Giulio Malavolta\thanks{Bocconi University, BIDSA. \texttt{giulio.malavolta@unibocconi.it}.} \and Alon Rosen \thanks{Bocconi University, BIDSA. \texttt{alon.rosen@unibocconi.it}.}}
\date{}

\begin{document}

\maketitle

\begin{abstract}
    Motivated by polynomial identity testing with exponentials (Li and Wu, ITCS'26), we study uncertainty principles for the number-theoretic transform (NTT). We show that the NTT satisfies strong sparsity tradeoffs: For every fixed prime $q$ and for all but finitely many primes $p \equiv 1 \pmod q$ every nonzero $f\in \mathbb F_p^{\Z_q}$ and its number-theoretic transform $\hat f$ satisfy
\[
|\mathrm{Supp}(f)| + |\mathrm{Supp}(\hat f)| \ge q+1.
\]
Thus, a $k$-sparse function has transform support at least $q-k+1$. As our main technical contribution, we prove a \emph{probabilistic} version of the above uncertainty principle, averaged over primes $p$, in the regime $p=q^{O(1)}$.

As an application, we obtain a black-box identity test for $k$-sparse exponential polynomials of degree at most $d$ with vanishing soundness error, for $q$ moderately larger than $k$.
\end{abstract}

%\tableofcontents

\section{Introduction}
Polynomial identity testing (PIT) asks whether a given algebraic expression is identically zero.
For ordinary low-degree polynomials over a field, this problem is well understood: One can
evaluate the polynomial at a random point, and the Schwartz--Zippel lemma \cite{schwartz1980fast,zippel1979probabilistic,demillo1977probabilistic} shows that every
nonzero degree-\(d\) polynomial vanishes with probability at most \(d/|S|\) over a finite test set \(S\).

In a recent work \cite{li2026identity} Li and Wu introduced the analogous problem for expressions that contain exponentials.  A
typical object has the form
\[
P(x)=\sum_{j=1}^k f_j(x)\exp\left(\frac{g_j(x)}{h_j(x)}\right),
\]
where the \(f_j,g_j,h_j \in \mathbb{Z}[x]\) are polynomials. The identity-testing problem is to decide whether \(P\) is identically zero on its domain.  This differs from standard PIT in two basic ways: First, the expression is no longer a polynomial, so the Schwartz--Zippel lemma does not apply directly. Second, there is no canonical notion of evaluating \(\exp(\cdot)\) over a finite field, so even defining a useful black-box testing model requires care.

This problem is motivated by recent advances in machine-learning.  In optimization compilers for tensor programs, such as Mirage \cite{wu2025mirage} one often needs to test whether two subprograms are functionally equivalent. For purely algebraic fragments, randomized finite-field testing is already useful. The next
natural step is to handle fragments that involve exponentials, since operations such as
normalization and softmax introduce this type of structure. Li and Wu \cite{li2026identity} formalized a finite-field
evaluation model for such expressions by using two primes $p$ and $q$ with $p \equiv 1 \pmod{q}$: Arithmetic inside the
exponent is carried out modulo \(q\), arithmetic outside the exponent is carried out modulo \(p\),
and exponentials are interpreted using the unique subgroup of order \(q\) in \(\mathbb F_p^*\).  This gives
a mathematically clean black-box model and matches the algebraic structure used by the verifier in Mirage \cite{wu2025mirage}.

The central challenge is in proving soundness. The main result of \cite{li2026identity} shows that, for a sum of \(k\)
 exponential terms of degree \(d\), the acceptance probability on a nonzero polynomial  is at most $q^{-1/(k-1)} + O(dk^2/q)$. In the setting of interest where $q =k^{O(1)}$, the soundness error is therefore $1-\Theta({\log q}/{k})$. It is left as an open problem to prove a constant soundness error for the same test, which is conjectured to hold in light of empirical evidence \cite{wu2025mirage}.

 Li and Wu \cite{li2026identity} suggested that this conjecture is related to \emph{uncertainty principles} over finite fields. More specifically, let  \(G\leq \mathbb F_p^*\) be the unique subgroup of order \(q\), they define a linear map from polynomials modulo \(x^q-1\) to their evaluations on \(G\). In cryptography, this is known as (the inverse of) the \emph{number-theoretic transform} (NTT) which, among other things, is routinely used for fast implementations of lattice-related cryptographic schemes \cite{HoffsteinPipherSilverman1998NTRU,lyubashevsky2008swifft, lyubashevsky2013toolkit}. The conjecture can be interpreted as a \emph{strong} uncertainty principle of the NTT and it can be equivalently stated as follows:
 For sufficiently large $q$ and $p$, any univariate non-zero polynomial
\[
f(z)=\sum_{j=1}^k \beta_j z^{\alpha_j},
\]
with $k \leq q^{O(1)}$ distinct non-zero coefficients, has at most $O(q)$ roots in $G$. A formal version of this conjecture is stated in \Cref{conj:descartes}.

\subsection{Our Results}

Motivated by the above algorithmic applications, we prove a series of uncertainty principles for the NTT. Let $f \in \mathbb F_p^{\mathbb Z_q}$ and let $\hat f$ be its NTT, with $\mathrm{Supp}(f)$ denote the non-zero coordinates of $f$. Our first result is a \emph{strong} uncertainty principle for the NTT for infinitely many $p$ and $q$.

\begin{theorem}[Additive Uncertainty, Informal]
    For every fixed prime $q$, for all but finitely many primes \(p\equiv 1 \pmod q\), every non-zero \(f\in \mathbb F_p^{\Z_q}\) satisfies
    \[
    |\mathrm{Supp}(f)|+|\mathrm{Supp}(\hat f)|\ge q+1.
    \]
\end{theorem}
A formal version of is given in \Cref{thm:additive}. This implies, as an immediate corollary, a non-uniform version of \Cref{conj:descartes}, in the settings where $p$ is allowed to grow with $q$. In fact, we obtain a much stronger and essentially optimal bound in these settings.

We stress that the well-known Donoho–Stark \cite{donoho1989uncertainty,wigderson2021uncertainty} \emph{multiplicative} uncertainty principle over finite fields (i.e., where the bound is on the product of the cardinalities of the supports) does \emph{not} yield any meaningful bound in our context. To bear any implication on the soundness of the identity tester, one needs \emph{additive} uncertainty principles, such as \Cref{thm:additive}.

A limitation of the above result is that it does not offer a good control on the growth of $p$ with respect to $q$. A quantitative analysis shows that the theorem holds for any $p$ larger than $\approx q^{q^2}$, which is quite far from what is used in practice \cite{wu2025mirage}. In the more realistic settings where $p = q^{O(1)}$, \Cref{thm:additive} does not apply.

Motivated by this, we prove a weaker \emph{approximate} uncertainty relation, which also applies to the settings where $p$ is polynomially bounded in $q$. Such a relation holds for any fixed input vector, but only on average over the random choice of $p$ from the appropriate interval. The following theorem is an informal version of \Cref{thm:one-query-uncertainty}.

\begin{theorem}[Approximate Uncertainty, Informal]
    There exist infinitely many primes \(q\) such that, for any non-zero $f \in \Z^{\Z_q}$ with $|\Supp(f)|<q$, then
\[
   \E_{p\leq q^4}\bigl(|\Supp(\hat{f})|\bigr)
   \ge
   q-1-\frac{32(q-1)\log q}{q^2}\log_2 \norm{f}_1.
\]    
\end{theorem}
While this bound does not have any direct implication on the tester proposed in \cite{li2026identity}, we show how a comparatively simple modification of their test allows us to prove much stronger soundness bounds. In short, if one is willing to additionally randomize the choice of $p$ (sampled from the appropriate interval), then one obtains a test where every nonzero input is accepted with probability at most $\approx {dk^2}/{q}$. In particular, once \(q\) is moderately larger than \(dk^2\), the test has constant soundness.
\begin{theorem}[PIT with Exponentials, Informal]
   There exists a randomized test such that, for all integer-coefficient \(k\)-sparse exponential polynomials \(P\not\equiv 0\) of degree at most \(d\) and coefficients bounded by $w$, the test accepts with probability at most
\[
   O\!\left(\frac{dk^2}{q}\right)
   +O\!\left(\frac1q\right)
   +O\!\left(\frac{\log(kw)\log q}{q^2}\right).
\]
\end{theorem}
We refer the reader to \Cref{thm:one-query-pit} for a precise statement. We highlight that, if $q=k^C$ with $C>2$ and $kw$ is polynomially bounded, then the soundness error is $o(1)$, directly improving over \cite{li2026identity}, where the bound is dominated by the term $q^{-1/(k-1)}=1-o(1)$.

Besides testing  equivalence of program fragments that contain exponentials, this test may be useful for verifying arithmetic circuits with exponentiation gates and may provide a soundness primitive for proof systems built for such circuits.

\subsection{Proof Overview}

Recall that the number-theoretic transform (NTT) is the finite-field analogue of the discrete Fourier transform. Fix primes \(q\) and \(p\) with \(p\equiv 1 \pmod{q}\), and let \(g\in \mathbb F_p^*\) have order \(q\). Then evaluation on the subgroup $G=\{1,g,g^2,\dots,g^{q-1}\}$ defines an \(\mathbb F_p\)-linear isomorphism $E:\mathbb F_p[x]/(x^q-1)\mapsto \mathbb F_p^q$
\[
E(f) = \bigl(f(1),f(g),\dots,f(g^{q-1})\bigr)
\]
whose inverse \(\mathcal F=E^{-1}\) is the NTT. Writing \(f(x)=\sum_{j=0}^{q-1}c_jx^j\), the evaluation map is represented by the Vandermonde matrix
\[
M=(g^{tj})_{t,j\in \mathbb Z_q},
\]
so uncertainty statements for the NTT naturally translate to statements about sparse polynomials and their evaluations on \(G\). A good strategy in this context \cite{tao2005uncertainty}, is to show \emph{hyperinvertibility} of $M$, i.e., that every square submatrix of \(M\) must be invertible. A routine calculation shows that this implies the strong (additive) uncertainty principle of the associated Fourier transform.

When the co-domain of the map is $\mathbb{C}$, one can show that $M$ is hyperinvertible \cite{tao2005uncertainty}, so we seek a finite-field analogue of this statement. The main difficulty is that such hyperinvertibility does not automatically translate over finite fields, and the direct analogue is in fact false \cite{zhang2019chebotarev}. Fortunately, a recent result of Caragea et al.\ \cite{caragea2025principal} shows that, if $p$ is sufficiently large compared to $q$, then the hyperinvertibility of $M$ holds. We can directly use this fact to establish our first uncertainty principle. In \Cref{sec:additive}, we also provide a somewhat more direct proof of this fact, using properties of the resultant of polynomials.

Yet the above proof is not useful for algorithmic applications, since a conservative bound on the size of $p$ results into $p \approx q^{q^2}$, which is not realistic. To handle the regime where $p = q^{O(1)}$, we instead prove a weaker \emph{probabilistic} property, averaging over primes $p$ over a bounded interval. The idea is as follows, fix a nonzero

\[
f(z) = \sum_{j\in\mathbb Z_q} c_j z^j\in \mathbb Z[z]
\]
with \(|\Supp(f)|<q\), then $f$ cannot vanish at a primitive \(q\)-th root of unity: Otherwise the cyclotomic polynomial \(\Phi_q(z)=1+z+\cdots+z^{q-1}\) would divide \(f\), forcing \(f\) to be a scalar multiple of \(\Phi_q\), which has full support \(q\). It follows that their resultant $\Res(f,\Phi_q)$ is a nonzero integer. Poisson's formula for the resultant gives
\[
|\Res(f,\Phi_q)|=\prod_{b\in\mathbb Z_q^*}|f(\omega^b)|\le \norm{f}_1^{q-1}.
\]
If \(f\) has an order-\(q\) root modulo some prime \(p\equiv1\pmod q\), then \(f\) and \(\Phi_q\) have a common root modulo \(p\), and hence \(p\mid \Res(f,\Phi_q)\). Thus there are at most \((q-1)\log_2 \norm{f}_1\) such bad primes.
To make this useful in the polynomial-size regime, we sample \(p\) from the interval
\[
\mathcal P(q)=\{p\le q^4:\ p\text{ prime and }p\equiv1\pmod q\}.
\]
We then appeal to the Bombieri--Vinogradov theorem \cite{bombieri1974grand,vinogradov1965density} to show that there are infinitely many primes \(q\) for which
\[
|\mathcal P(q)|\ge \frac{q^3}{32\log q}.
\]
Combining this estimate with the above bound, we obtain our approximate uncertainty principle.  

This directly suggests a natural one-query zero-tester for polynomials \(P\) with exponentiations. Given a polynomial $P$ we sample:
\begin{itemize}
    \item A random admissible prime \(p\in\mathcal P(q)\).
    \item A random point \(v\in\mathbb F_q^n\) for the exponents.
    \item A random point \(u\in\mathbb F_p^n\) for the coefficients.
    \item A generator \(g\) of the order-\(q\) subgroup \(G\leq\mathbb F_p^*\).
    \item A random \(r\in\mathbb Z_q\), and set \(a=g^r\).
\end{itemize}
We then evaluate the algebraic extension
\[
P_{p,q,a}(u,v)
=
\sum_{j=1}^k f_j(u)\,
a^{g_j(v)h_j(v)^{-1}},
\]
declaring the value undefined if some denominator vanishes, and accept iff the value is either \(0\) or undefined.

The soundness proof is as follows. First, by Schwartz--Zippel and a union bound, we bound the probability that all denominators are nonzero and the exponents $\alpha_j=g_j(v)h_j(v)^{-1}$
are pairwise distinct. Conditioned on this event, the test evaluates the ordinary polynomial
\[
H_{a,v}(X)=\sum_{j=1}^k f_j(X)a^{\alpha_j}
\]
at a random point \(u\in\mathbb F_p^n\). If \(H_{a,v}\not\equiv0\), Schwartz--Zippel gives error at most \(d/(q+1)\). If \(H_{a,v}\equiv0\), then some nonzero monomial coefficient vector \(\eta=([X^\mu]f_1,\ldots,[X^\mu]f_k)\), with \(\|\eta\|_1\le kw\), satisfies
\[
\sum_{j=1}^k \eta_j g^{r\alpha_j}=0.
\]
This is exactly the event that a random coordinate of the evaluation vector of a sparse integer polynomial with \(\ell_1\)-norm at most \(kw\) vanishes. The approximate uncertainty principle is precisely a bound on the probability of this event. The bound on the soundness error of the test is then obtained by combining the above estimates.

\section{Preliminaries}
We write $\Z_q$ for $\Z/q\Z$.  If $S$ is finite, $x\sim\Unif(S)$ means that $x$ is sampled uniformly from $S$.  For a vector $v=(v_a)_{a\in A}$, write
\[
   \Supp(v):=\{a\in A:v_a\ne0\}.
\]
All logarithms are natural unless the base is displayed explicitly.

\subsection{Resultants}

We recall the elementary properties of the resultant that will be used throughout.  If $K$ is a field and $f,g\in K[z]$ are nonzero, then
\begin{equation}\label{eq:res-gcd}
   \Res(f,g)=0
   \quad\Longleftrightarrow\quad
   f\text{ and }g\text{ have a common root over }\overline K.
\end{equation}
Equivalently, $\gcd(f,g)\ne1$ in $K[z]$.  If $g$ is monic, Poisson's product formula gives
\begin{equation}\label{eq:poisson}
   |\Res(f,g)|=\prod_{\alpha:g(\alpha)=0}|f(\alpha)|,
\end{equation}
where the product is taken in an algebraic closure, with multiplicities.  

\begin{lemma}[Modular Reduction]\label{lem:resultant-functoriality}
Let \(f,g\in\Z[z]\), and let \(p\) be prime.  If the reductions
\(f_p,g_p\in\F_p[z]\) have a common root over \(\overline{\F}_p\), then $p\mid \Res(f,g)$.
\end{lemma}

\begin{proof}
Let \(m=\deg f\) and \(n=\deg g\), and form the Sylvester matrix of \(f\) and \(g\) using these original degrees.  Its determinant is \(\Res(f,g)\).  Reducing this matrix modulo \(p\) gives the corresponding Sylvester matrix with the same prescribed sizes for \( f_p\) and \(g_p\), possibly with zero leading coefficients if degrees drop.

If \( f_p\) and \( g_p\) have a common root over \(\overline{\F}_p\), then they are not coprime in \(\F_p[z]\).  Hence there exist nonzero polynomials \(A,B\in\F_p[z]\), with \(\deg A<n\) and \(\deg B<m\), such that
\[
   A f_p+B g_p=0.
\]
Equivalently, the reduced Sylvester matrix has a nontrivial kernel.  Its determinant is therefore zero in \(\F_p\), i.e., $\Res(f,g)\equiv0\pmod p$.
\end{proof}

\subsection{Schwartz--Zippel}

We recall the Schwartz--Zippel lemma \cite{schwartz1980fast,zippel1979probabilistic,demillo1977probabilistic}.

\begin{lemma}[Schwartz--Zippel]\label{lem:schwartz-zippel}
Let $\mathbb F$ be a field, let $H\in \mathbb F[x_1,\ldots,x_n]$ be a nonzero polynomial of total degree at most $d$, and let $S\subseteq \mathbb F$ be finite and nonempty.  Then
\[
   \Pr_{u\sim\Unif(S^n)}\left(H(u)=0\right)\le \frac{d}{|S|}.
\]
\end{lemma}

\subsection{The Number-Theoretic Transform}\label{sec:ntt}

Let $q$ be prime, let $p$ be prime with $q\mid p-1$, and choose an element $\gamma\in\F_p^\ast$ of exact order $q$.  The order-$q$ subgroup is
\[
   G_{p,q}:=\{1,\gamma,\gamma^2,\ldots,\gamma^{q-1}\}\le\F_p^\ast.
\]
Since $G_{p,q}$ is the unique subgroup of order $q$, it is independent of the choice of $\gamma$.
Define the evaluation map
\begin{equation}\label{eq:evaluation-map}
   E=E_{p,q,\gamma}:\F_p[x]/(x^q-1)\mapsto\F_p^{\Z_q},
   \qquad
   E(f):=\bigl(f(\gamma^t)\bigr)_{t\in\Z_q}.
\end{equation}
If $f(x)=\sum_{c\in\Z_q} a_c x^c$, with exponents represented by $0,\ldots,q-1$, then
\[
   E(f)(t)=\sum_{c\in\Z_q} a_c\gamma^{tc}.
\]
Thus, relative to the monomial basis on the domain and the standard basis on the codomain, $E$ is represented by the Fourier matrix
\begin{equation}\label{eq:fourier-matrix}
   M=M_{p,q,\gamma}:=(\gamma^{tc})_{t,c\in\Z_q}.
\end{equation}

\begin{proposition}\label{prop:evaluation-isomorphism}
The map $E$ in \Cref{eq:evaluation-map} is a well-defined $\F_p$-linear isomorphism.
\end{proposition}

\begin{proof}
Because $\gamma^q=1$, evaluation at $\gamma^t$ depends only on the class of a polynomial modulo $x^q-1$.  Moreover,
\[
   x^q-1=\prod_{t\in\Z_q}(x-\gamma^t)
\]
splits over $\F_p$ into distinct linear factors.  The Chinese remainder theorem gives
\[
   \F_p[x]/(x^q-1)
   \cong
   \prod_{t\in\Z_q}\F_p[x]/(x-\gamma^t)
   \cong
   \F_p^{\Z_q},
\]
and this is exactly the evaluation map $E$.
\end{proof}

\begin{definition}[NTT]
The number-theoretic transform (NTT) associated with $(p,q,\gamma)$ is
\[
   \mathcal F:=E^{-1}:\F_p^{\Z_q}\mapsto \F_p[x]/(x^q-1).
\]
When convenient, we identify $\mathcal F(v)$ with its coefficient vector in $\F_p^{\Z_q}$.
\end{definition}
The inverse is explicit.  For $v\in\F_p^{\Z_q}$ and $c\in\Z_q$,
\begin{equation}\label{eq:inverse-ntt}
   (\mathcal Fv)_c
   =
   \frac1q\sum_{t\in\Z_q} v_t\gamma^{-tc}
\end{equation}
where $q^{-1}$ is taken in $\F_p$.  This follows from the orthogonality relation
\[
   \sum_{t\in\Z_q}\gamma^{t(c-c')}=
   \begin{cases}
   q& \text{ if }c=c'\\
   0& \text{ if }c\ne c'.
   \end{cases}
\]

\section{Uncertainty Principles}

For $v\in\F_p^{\Z_q}$, $\Supp(v)$ denotes the nonzero coordinates of $v$.  For $f=\sum_{c\in\Z_q}a_cx^c\in\F_p[x]/(x^q-1)$, $\Supp(f):=\{c:a_c\ne0\}$.  Under our identification of $\mathcal Fv$ with its coefficient vector, $\Supp(\mathcal Fv)$ is well-defined.

\subsection{Additive Uncertainty}\label{sec:additive}

We prove the first version of the  additive uncertainty principle in the following: We show that there are infinitely many pairs of primes $p$ and $q$ that satisfy the stronger (additive) uncertainty principle.

\begin{theorem}[Additive uncertainty]\label{thm:additive}
Fix a prime $q$.  For all but finitely many primes $p$ with $p\equiv1\pmod q$, the following holds.  For every element $\gamma\in\F_p^\ast$ of order $q$ and every nonzero $v\in\F_p^{\Z_q}$,
\[
   |\Supp(v)|+|\Supp(\mathcal Fv)|\ge q+1.
\]
\end{theorem}
To prove \Cref{thm:additive} we adopt the strategy from \cite{tao2005uncertainty}, where it is shown that additive uncertainty is equivalent to invertibility of every square submatrix of the Fourier matrix $M$ (see \Cref{sec:ntt}). This statement was shown in the complex settings, but the same proof works for finite-fields.

\begin{lemma}[Hyperinvertibility and Additive Uncertainty]\label{lem:additive-full-spark}
Let $q$ and $p$ be primes with $q\mid p-1$, let $\gamma\in\F_p^\ast$ have order $q$, and let $M=(\gamma^{tc})_{t,c\in\Z_q}$.  The following are equivalent.
\begin{enumerate}
   \item[(i)] Every square submatrix of $M$ is nonsingular over $\F_p$.
   \item[(ii)] Every nonzero $v\in\F_p^{\Z_q}$ satisfies $|\Supp(v)|+|\Supp(\mathcal Fv)|\ge q+1$.
\end{enumerate}
\end{lemma}
\begin{proof}
Assume first that every square submatrix is nonsingular.  Let $w:=\mathcal Fv$, and set $S:=\Supp(w)$, $k:=|S|$.  Write $v=Mw$.  Let $Z:=\Z_q\setminus\Supp(v)$
be the zero set of $v$.  Restricting the identity $v=Mw$ to rows in $Z$ and columns in $S$ gives
\[
   0=M_{Z,S}w_S.
\]
If $|Z|\ge k$, choose $Z'\subseteq Z$ with $|Z'|=k$.  Then $M_{Z',S}$ is a nonsingular $k\times k$ matrix, so $w_S=0$, contradicting the definition of $S$ unless $v=0$.  Hence $|Z|\le k-1$, which is equivalent to
\[
   q-|\Supp(v)|\le |\Supp(\mathcal Fv)|-1.
\]
Conversely, suppose that some $k\times k$ submatrix $M_{R,S}$ is singular.  Choose $0\ne a\in\F_p^S$ with $M_{R,S}a=0$, extend it by zero outside $S$, and set $v:=Ma$.  Then $\mathcal Fv=a$, so $|\Supp(\mathcal Fv)|\le k$, while $v$ vanishes on all rows in $R$, so $|\Supp(v)|\le q-k$.  Thus
\[
   |\Supp(v)|+|\Supp(\mathcal Fv)|\le q,
\]
contradicting additive uncertainty.
\end{proof}
We shall use the following classical theorem of Chebotarev, in the form used in \cite{tao2005uncertainty}.

\begin{theorem}[Hyperinvertibility of Fourier Matrices]\label{thm:tao-chebotarev}
Let $q$ be prime and let $\omega=e^{2\pi i/q}$.  For every $1\le k\le q$ and every pair of $k$-element subsets $R,C\subseteq\Z_q$,
\[
   \det(\omega^{rc})_{r\in R,c\in C}\ne0.
\]
\end{theorem}
One may hope to prove a finite-field analogue of \Cref{thm:tao-chebotarev}, unfortunately such a result cannot be true. It was shown by Zhang \cite{zhang2019chebotarev} that in the finite field settings there exists a Fourier matrix with a singular submatrix. 

To circumvent this obstacle, we show that the finite-field analogue of \Cref{thm:tao-chebotarev} holds for all but finitely-many primes $p$. We shall mention that this fact also follows as a consequence of a theorem proven in a very recent work 
\cite{caragea2025principal}. Here we choose to give an explicit proof, since it serves as a good warm-up for subsequent arguments.

\begin{lemma}[Finite Field Hyperinvertibility]\label{lem:finite-exceptional-set}
Fix a prime $q$.  There is a finite set $B_q$ of primes such that, whenever $p\equiv1\pmod q$ and $p\notin B_q$, every square submatrix of $M_{p,q,\gamma}$ is nonsingular over $\F_p$ for every element $\gamma\in\F_p^\ast$ of order $q$.
\end{lemma}

\begin{proof}
Fix $1\le k\le q$ and $k$-element sets $R,C\subseteq\Z_q$.  Choose integer representatives in $\{0,\ldots,q-1\}$ and define
\[
   D_{R,C}(z):=\det(z^{rc})_{r\in R,c\in C}\in\Z[z],
\]
where each exponent $rc$ is represented modulo $q$ by an integer in $\{0,\ldots,q-1\}$.  Let
   $\Phi_q(z):=1+z+\cdots+z^{q-1}$
be the $q$-th cyclotomic polynomial.
For every $b\in\Z_q^\ast$, multiplication by $b$ permutes $\Z_q$, so $bC$ is again a $k$-element subset of $\Z_q$.  By \Cref{thm:tao-chebotarev},
\[
   D_{R,C}(\omega^b)=\det(\omega^{brc})_{r\in R,c\in C}
   =\det(\omega^{r c'})_{r\in R,c'\in bC}
   \ne0.
\]
Thus $D_{R,C}$ and $\Phi_q$ have no common complex root, and $\Res(D_{R,C},\Phi_q)$
is a nonzero integer.
Let
\[
   \Delta_q:=\prod_{k=1}^q\prod_{\substack{R,C\subseteq\Z_q\\ |R|=|C|=k}}\Res(D_{R,C},\Phi_q)\in\Z\setminus\{0\},
\]
and let $B_q$ be the set of prime divisors of $\Delta_q$. Suppose $p\equiv1\pmod q$, $p\notin B_q$, and $\gamma\in\F_p^\ast$ has order $q$.  Then $\Phi_q(\gamma)=0$ in $\F_p$.  If some minor $M_{R,C}$ were singular, then
\[
   D_{R,C}(\gamma)=0\quad\text{in }\F_p.
\]
Hence the reductions of $D_{R,C}$ and $\Phi_q$ modulo $p$ would have the common root $\gamma$.  By \Cref{lem:resultant-functoriality}, this implies
\[
   \Res(D_{R,C},\Phi_q)\equiv0\pmod p,
\]
so $p\mid \Delta_q$, a contradiction.
\end{proof}
The proof of our main theorem is now straightforward.
\begin{proof}[Proof of \Cref{thm:additive}]
For fixed $q$, take $B_q$ from \Cref{lem:finite-exceptional-set}.  If $p\equiv1\pmod q$ and $p\notin B_q$, the matrix $M_{p,q,\gamma}$ is hyperinvertible for every order-$q$ element $\gamma$. \Cref{lem:additive-full-spark} gives the claimed uncertainty inequality.
\end{proof}

\subsection{Approximate Uncertainty}
The proof above is qualitative because it multiplies all submatrices.  We next keep a single resultant at a time and bound the number of primes that can divide it.

\begin{lemma}[Bad Primes]\label{lem:bad-prime-principle}
Let $q$ be prime, let $F(z)\in\Z[z]$, and $\omega=e^{2\pi i/q}$.  Suppose that, for some real $B\ge1$,
\[
   F(\omega^b)\ne0
   \quad\text{and}\quad
   |F(\omega^b)|\le B
   \qquad\text{for every }b\in\Z_q^\ast.
\]
Then the number of primes $p\equiv1\pmod q$ for which there exists an element $\gamma\in\F_p^\ast$ of order $q$ with
\[
   F(\gamma)=0\quad\text{in }\F_p
\]
is at most $(q-1)\log_2 B$.
\end{lemma}

\begin{proof}
The hypotheses imply that $F$ and $\Phi_q$ have no common complex root.  Hence
\[
   \Res(F,\Phi_q)\in\Z\setminus\{0\}.
\]
If $p\equiv1\pmod q$ and $F(\gamma)=0$ for an element $\gamma\in\F_p^\ast$ of order $q$, then also $\Phi_q(\gamma)=0$.  Thus the reductions of $F$ and $\Phi_q$ modulo $p$ have a common root.  By \Cref{lem:resultant-functoriality}, $p\mid \Res(F,\Phi_q)$.

By \Cref{eq:poisson} (Poisson's formula),
\[
   \abs{\Res(F,\Phi_q)}=\prod_{b\in\Z_q^\ast}\abs{F(\omega^b)},
\]
and therefore $|\Res(F,\Phi_q)|\le B^{q-1}$.  If a nonzero integer $N$ has $m$ distinct prime divisors, then $2^m\le |N|$.  Consequently $\Res(F,\Phi_q)$ has at most $\log_2(B^{q-1})=(q-1)\log_2B$ distinct prime divisors.
\end{proof}
Let
\begin{equation}\label{eq:Pq-def}
   \mathcal{P}(q):=\{p\le q^4:p\text{ prime and }p\equiv1\pmod q\}.
\end{equation}
We shall use primes $q$ for which
\begin{equation}\label{eq:many-primes}
   |\mathcal{P}(q)|=\pi(q^4;q,1)\ge \frac{q^3}{32\log q}.
\end{equation}
We prove an approximate uncertainty principle for a fixed integer sparse vector, but averaged over the random choice of primes $p \in \mathcal{P}(q)$.
To prove the approximate uncertainty principle, all needs to be done is to show that infinitely many such $q$ exist. Here we simply postulate that such $q$ exists, and we refer to \Cref{sec:bounds} for a proof of this fact.
\begin{theorem}[Approximate Uncertainty]\label{thm:one-query-uncertainty}
Let $q$ be prime and suppose \Cref{eq:many-primes} holds.  Let $0\ne c\in\Z^{\Z_q}$ satisfy
\[
   |\Supp(c)|<q
   \qquad\text{and}
   \qquad\|c\|_1:=\sum_{a\in\Z_q}|c_a|\le W
\]
where $W\ge2$. Then 
\[
   \E_p\bigl(|\Supp(\mathcal F^{-1}(c))|\bigr)
   \ge
   q-1-\frac{32(q-1)\log q}{q^2}\log_2 W.
\]
\end{theorem}

\begin{proof}
Let $C(z):=\sum_{a\in\Z_q} c_a z^a\in\Z[z]$. We first show that $C$ does not vanish at any primitive $q$-th root of unity. If $C(\omega)=0$ for a primitive root $\omega$, then the minimal polynomial $\Phi_q$ of $\omega$ over $\Q$ divides $C$ in $\Q[z]$.  Since $\deg C\le q-1=\deg\Phi_q$, this forces $C=\lambda\Phi_q$ for some $\lambda\in\Q$.  But $\Phi_q=1+z+\cdots+z^{q-1}$ has support $q$, whereas $0\ne C$ has support strictly smaller than $q$.  This is impossible.  The same argument applies to every primitive root $\omega^b$, for $b\in\Z_q^\ast$.

Moreover for all $b\in\Z_q^\ast$,
\[
   |C(\omega^b)|\le\sum_a |c_a|\le W.
\]
By \Cref{lem:bad-prime-principle}, the number of primes $p\equiv1\pmod q$ for which there exists an order-$q$ element $h\in\F_p^\ast$ with $C(h)=0$ in $\F_p$ is at most $(q-1)\log_2W$.  Call these primes bad.

If $p\in\mathcal P(q)$ is not bad, then for every $t\in\Z_q^\ast$, the element $\gamma_p^t$ has order $q$, and hence
\[
   \mathcal{F}^{-1}(c)(t) = E_{p,q,\gamma_p}(c)(t)=C(\gamma_p^t)\ne0.
\]
Thus, for non-bad $p$, the evaluation vector can vanish only at $t=0$, and the zero fraction is at most $1/q$.  For bad $p$ we use the trivial bound $1$.  Therefore
\[
   \E_p\left(
   \frac1q\bigl|\{t:E_{p,q,\gamma_p}(c)(t)=0\}\bigr|
   \right)
   \le
   \frac1q+\frac{(q-1)\log_2W}{|\calP(q)|}.
\]
Using \Cref{eq:many-primes} gives the claimed bound.  The desired statement follows by multiplying the zero-fraction inequality by $q$ and subtracting from $q$.
\end{proof}

\subsection{Bounds on Admissible Primes}\label{sec:bounds}
For $n\ge1$, let $\Lambda(n)$ denote the von Mangoldt function.  For $(a,q)=1$, define
\[
   \psi(x;q,a):=\sum_{\substack{n\le x\\ n\equiv a\, (q)}}\Lambda(n),
   \qquad
   \theta(x;q,a):=\sum_{\substack{p\le x\\ p\equiv a\, (q)}}\log p.
\]
We use the following standard form of Bombieri--Vinogradov \cite{bombieri1974grand,vinogradov1965density,vaughanbombieri}.

\begin{theorem}[Bombieri--Vinogradov]\label{thm:bombieri-vinogradov}
For every $A>0$ there are constants $B=B(A)$ and $C=C(A)$ such that, for all $x\ge2$ and all $Q\le x^{1/2}(\log x)^{-B}$,
\[
   \sum_{q\le Q}
   \max_{(a,q)=1}
   \max_{2\le y\le x}
   \left|\psi(y;q,a)-\frac{y}{\varphi(q)}\right|
   \le
   C\frac{x}{(\log x)^A}.
\]
\end{theorem}
The next lemma shows that there are infinitely many $q$ with the desired property.
\begin{lemma}[Admissible Primes]\label{lem:many-primes}
There are infinitely many primes $q$ satisfying \Cref{eq:many-primes}.
\end{lemma}

\begin{proof}
First compare $\psi$ and $\theta$. For $y\ge3$,
\begin{align*}
   0\le \psi(y;q,1)-\theta(y;q,1)
   &\le
   \sum_{m=2}^{\lfloor\log_2 y\rfloor}\sum_{p\le y^{1/m}}\log p  \\
   &\le
   (\log y)\sum_{m=2}^{\lfloor\log_2 y\rfloor}\frac{y^{1/m}}{m},
\end{align*}
using the elementary bound $\sum_{p\le t}\log p\le t\log t$.  With $y=q^4$, this gives, for $q\ge3$,
\begin{equation}\label{eq:psi-theta}
   0\le \psi(q^4;q,1)-\theta(q^4;q,1)
   \le
   2q^2\log q+4q^{4/3}\log q\,(1+\log(4\log q)).
\end{equation}
Indeed, the term $m=2$ contributes at most $2q^2\log q$, and the remaining terms are bounded by
\[
   4q^{4/3}\log q\sum_{m=3}^{\lfloor\log_2(q^4)\rfloor}\frac1m
   \le
   4q^{4/3}\log q\,(1+\log(4\log q)).
\]
Fix $A>2$, and let $B,C$ be as in \Cref{thm:bombieri-vinogradov}.  For a large parameter $N$, set
\[
   x:=(2N)^4,
   \qquad
   Q:=x^{1/2}(\log x)^{-B}.
\]
For all sufficiently large $N$, the interval $[N,2N]$ is contained in $[1,Q]$.  Let $E_N$ be the set of integers $q\in[N,2N]$ for which
\[
   \left|\psi(q^4;q,1)-\frac{q^4}{\varphi(q)}\right|>\frac{q^4}{2\varphi(q)}.
\]
For $q\in[N,2N]$, the right-hand side is $\gg N^3$. Bombieri--Vinogradov and Markov's inequality therefore give
\[
   |E_N|=O\!\left(\frac{N}{(\log N)^A}\right).
\]
By the prime number theorem, the interval $[N,2N]$ contains $\gg N/\log N$ primes.  Since $A>2$, for all sufficiently large $N$ there is a prime $q\in[N,2N]\setminus E_N$.  Taking $N\to\infty$ gives infinitely many such primes. For any such prime $q$, since $\varphi(q)=q-1$,
\[
   \psi(q^4;q,1)
   \ge
   \frac{q^4}{2(q-1)}
   \ge
   \frac{q^3}{4}
\]
for $q\ge2$.  Combining this with \Cref{eq:psi-theta}, and taking $q$ sufficiently large along the infinite sequence, gives
\[
   \theta(q^4;q,1)\ge \frac{q^3}{8}.
\]
Finally,
\[
   \theta(q^4;q,1)
   \le
   \pi(q^4;q,1)\log(q^4)
   =4\pi(q^4;q,1)\log q,
\]
so \Cref{eq:many-primes} follows.
\end{proof}

\subsection{Multiplicative Uncertainty}

We refer to the standard Donoho–Stark \cite{donoho1989uncertainty,wigderson2021uncertainty} uncertainty principle as the \emph{multiplicative} uncertainty principle. For completeness, we recall here the finite-field analogue of Donoho–Stark uncertainty and we refer the reader to \cite{borello2022uncertainty} for a proof. However, we wish to emphasize that this is not going to be useful for our work and we are instead interested in stronger \emph{additive} uncertainty.

\begin{theorem}[Multiplicative Uncertainty]\label{thm:multiplicative}
Let $q$ be prime, let $p$ be prime with $q\mid p-1$, and let $\mathcal F$ be the NTT above.  For every nonzero $v\in\F_p^{\Z_q}$,
\[
   |\Supp(v)|\cdot|\Supp(\mathcal Fv)|\ge q.
\]
\end{theorem}

\section{Applications to Exponential Polynomial Identity Testing}
\label{sec:pit}

We now apply the preceding uncertainty principles to the finite-field identity-testing model of Li and Wu \cite{li2026identity}, which we recall in the following.  Let
\begin{equation}\label{eq:P-def}
   P(x)=\sum_{j=1}^k f_j(x)
   \exp\!\left(\frac{g_j(x)}{h_j(x)}\right)
\end{equation}
with $f_j,g_j,h_j\in\Z[x_1,\ldots,x_n]$.  We say that $P$ is \emph{condensed} if each $f_j$ and $h_j$ is nonzero and
\begin{equation}\label{eq:condensed}
   g_jh_t-g_th_j\not\equiv0
   \qquad\text{for all }j\ne t.
\end{equation}
Thus the rational functions $g_j/h_j$ are pairwise distinct.
We say that $P$ has degree at most $d$ if each $f_j,g_j,h_j$ has total degree at most $d$.  We say that $P$ has height at most $w$ if all integer coefficients appearing in
\begin{equation}\label{eq:height-def}
   f_j,
   \qquad
   h_j,
   \qquad
   g_jh_t-g_th_j\quad(j\ne t)
\end{equation}
have absolute value at most $w$.  We shall assume $w\ge2$.

Let $p,q$ be primes with $p\equiv1\pmod q$, and let $G_{p,q}\le\F_p^\ast$ be the order-$q$ subgroup.  For $\gamma\in G_{p,q}$, $u\in\F_p^n$, and $v\in\F_q^n$, define the algebraic evaluation
\begin{equation}\label{eq:algebraic-evaluation}
   P_{p,q,\gamma}(u,v)
   :=
   \sum_{j=1}^k f_j(u)\,
   \gamma^{\,g_j(v)h_j(v)^{-1}}
   \in\F_p,
\end{equation}
where the exponent is computed in $\F_q\cong\Z_q$.  If some $h_j(v)$ is zero in $\F_q$, the value is declared to be $\bot$.

\subsection{A Strong Descartes Rule over Finite Fields}\label{sec:descartes}

Li and Wu \cite{li2026identity} formulate the following finite-field analogue of Descartes' rule as a sufficient route to improved soundness.

\begin{conjecture}[Strong finite-field Descartes rule]\label{conj:descartes}
There exist constants $\eps,\delta<1$ such that the following holds.  Let $p,q$ be sufficiently large primes with $p\equiv1\pmod q$, let $\gamma\in\F_p^\ast$ have order $q$, and let $G\le\F_p^\ast$ be the subgroup of order $q$.  If $k\le q^\delta$, if $\alpha_1,\ldots,\alpha_k\in\Z_q$ are distinct, and if $\beta_1,\ldots,\beta_k\in\F_p$, then
\[
   A(z):=\sum_{j=1}^k \beta_j z^{\alpha_j}
\]
has at most $\eps q$ roots in $G$, unless all $\beta_j$ are zero.
\end{conjecture}
Our additive uncertainty theorem gives a stronger conclusion for fixed $q$ and all but finitely many admissible $p$.

\begin{theorem}[Nonuniform Strong Descartes Rule]\label{thm:strong-descartes}
Fix a prime $q$.  For all but finitely many primes $p\equiv1\pmod q$, the following holds.  Let $\gamma\in\F_p^\ast$ have order $q$, and let $G\le\F_p^\ast$ be the subgroup of order $q$.  If
\[
   A(z)=\sum_{j=1}^k \beta_j z^{\alpha_j},
\]
where $\alpha_1,\ldots,\alpha_k\in\Z_q$ are distinct and $\beta_1,\ldots,\beta_k\in\F_p^\ast$, then $A$ has at most $k-1$ roots in $G$.
\end{theorem}

\begin{proof}
Choose $p$ so that \Cref{thm:additive} holds.  Equivalently, by \Cref{lem:additive-full-spark}, the matrix $M=(\gamma^{ta})_{t,a\in\Z_q}$ is hyperinvertible.
Let $S:=\{\alpha_1,\ldots,\alpha_k\}$ and define $c\in\F_p^{\Z_q}$ by $c_{\alpha_j}=\beta_j$ and $c_a=0$ for $a\notin S$.  Then
\[
   (Mc)_t=A(\gamma^t).
\]
If $A$ vanished at $k$ or more elements of $G$, there would be a $k$-element set $R\subseteq\Z_q$ such that $(Mc)_t=0$ for every $t\in R$.  Hence
\[
   M_{R,S}c_S=0.
\]
The submatrix $M_{R,S}$ is nonsingular, while $c_S\ne0$, a contradiction.
\end{proof}

\subsection{A Modified Polynomial Identity Test}\label{sec:one-query-test}

Assume $q$ satisfies \Cref{eq:many-primes}. We describe a modification to the identity tester described in \cite{li2026identity}, where we additionally randomize the choice of the prime $p$. The following test makes a single algebraic evaluation (see \Cref{eq:algebraic-evaluation}) of a fixed polynomial $P$. Sample:
\begin{itemize}
   \item $p\sim\Unif(\calP(q))$.
   \item $v\sim\Unif(\F_q^n)$.
   \item $u\sim\Unif(\F_p^n)$.
   \item An element $\gamma\in\F_p^\ast$ of order $q$.
   \item $r\sim\Unif(\Z_q)$ and set $a:=\gamma^r$.
\end{itemize}
Accept iff $P_{p,q,a}(u,v)\in\{0,\bot\}$.

\begin{theorem}[Soundness]\label{thm:one-query-pit}
Let $q$ be a prime satisfying \Cref{eq:many-primes}.  Let $P$ be a condensed $k$-sparse exponential polynomial of degree at most $d$ and height at most $w$, with $w\ge2$, and suppose $q>2kw$.
If $P\equiv0$, the one-query randomized-$p$ test accepts with probability $1$.  If $P\not\equiv0$, the test accepts with probability at most
\begin{equation}\label{eq:one-query-bound}
   \frac{kd}{q}
   +
   \frac{2d\binom{k}{2}}{q}
   +
   \frac{d}{q+1}
   +
   \frac1q
   +
   \frac{32(q-1)\log q}{q^3}\log_2(kw).
\end{equation}
\end{theorem}
\begin{proof}
    Completeness is immediate: If $P\equiv0$, every defined algebraic evaluation is zero, and undefined evaluations are accepted by convention.

Assume $P\not\equiv0$.  Define the good denominator event
\[
   E_h:=\{h_j(v)\ne0\text{ in }\F_q\text{ for all }j\in[k]\},
\]
and the good separation event
\[
   E_g:=\{g_j(v)h_t(v)-g_t(v)h_j(v)\ne0\text{ in }\F_q\text{ for all }j\ne t\}.
\]
Because $q>2kw\ge w$, every nonzero integer polynomial in \Cref{eq:height-def} remains nonzero after reduction modulo $q$.  By \Cref{lem:schwartz-zippel} (Schwartz--Zippel) and the union bound,
\begin{equation}\label{eq:bad-v-kquery}
   \Pr\bigl(\neg(E_h\cap E_g)\bigr)
   \le
   \frac{kd}{q}+\frac{2d\binom{k}{2}}{q}
\end{equation}
since $\deg(g_jh_t-g_th_j)\le2d$.

Condition on $E_h\cap E_g$ and set
\[
   \alpha_j:=g_j(v)h_j(v)^{-1}\in\Z_q.
\]
The $\alpha_j$ are pairwise distinct.  For fixed $p,\gamma,r$, with $a=\gamma^r$, the algebraic evaluation is the value at $u$ of
\[
   H_{a,v}(X):=\sum_{j=1}^k f_j(X)a^{\alpha_j}\in\F_p[X_1,\ldots,X_n].
\]
If $H_{a,v}\not\equiv0$, again \Cref{lem:schwartz-zippel} (Schwartz--Zippel) gives
\begin{equation}\label{eq:SZ-u-onequery}
   \Pr_{u\sim\Unif(\F_p^n)}\bigl(H_{a,v}(u)=0\bigr)
   \le
   \frac{d}{p}
   \le
   \frac{d}{q+1}.
\end{equation}
It remains to bound the probability (over the random choice of $p,\gamma,r$) that $H_{\gamma^r,v}\equiv0$.
Since $P\not\equiv0$ and is condensed, not all coefficient polynomials $f_j$ are zero.  Choose a monomial $X^\mu$ for which
\[
   \eta:=(\eta_1,\ldots,\eta_k)
   :=([X^\mu]f_1,\ldots,[X^\mu]f_k)
   \in\Z^k
\]
is nonzero.  The height bound gives
\[
   \sum_{j=1}^k |\eta_j|\le kw.
\]
Since every sampled $p\in\calP(q)$ satisfies $p\ge q+1>2kw$, the vector $\eta$ remains nonzero modulo $p$.
If $H_{\gamma^r,v}\equiv0$, then the coefficient of $X^\mu$ in this polynomial vanishes in $\F_p$:
\begin{equation}\label{eq:coeff-vanishes}
   \sum_{j=1}^k \eta_j\gamma^{r\alpha_j}=0.
\end{equation}
Define $c\in\Z^{\Z_q}$ by placing $\eta_j$ at the coordinate $\alpha_j$.
Because the $\alpha_j$ are distinct,
\[
   c\ne0,
   \qquad
   |\Supp(c)|\le k<q,
   \qquad
   \|c\|_1\le kw.
\]
Moreover,
\[
   E_{p,q,\gamma}(c)(r)=\sum_{j=1}^k \eta_j\gamma^{r\alpha_j}.
\]
Therefore \Cref{eq:coeff-vanishes} is the event $E_{p,q,\gamma}(c)(r)=0$. \Cref{thm:one-query-uncertainty}, applied with $W=kw$, gives
\begin{equation}\label{eq:H-identically-zero}
   \Pr_{p,r}\bigl(H_{\gamma^r,v}\equiv0\bigr)\le
   \Pr_{p,r}\bigl(E_{p,q,\gamma}(c)(r)=0\bigr)
   \le
   \frac1q+\frac{32(q-1)\log q}{q^3}\log_2(kw).
\end{equation}
Combining \Cref{eq:SZ-u-onequery},  \Cref{eq:H-identically-zero}, and \Cref{eq:bad-v-kquery} proves \Cref{eq:one-query-bound}.
\end{proof}

\subsection*{Acknowledgments}
We thank ChatGPT for helping with ideas, finding references, sanity-checking the math, and general discussions at any stage of this project. The proofs are all written by a human and the authors verified the correctness and originality of all content, including references. We also thank Jiatu Li and Mengdi Wu for valuable correspondence.

G.M.\ is supported by the European Research Council through an ERC Starting Grant (Grant agreement No.~101077455, ObfusQation) and by the Deutsche Forschungsgemeinschaft (DFG, German Research Foundation) under Germany's Excellence Strategy -- EXC 2092 CASA-390781972.

A.R.\ is supported by European Research Council (ERC) under the EU’s Horizon 2020 research and innovation programme (Grant agreement No.\ 101019547) and Cariplo CRYPTONOMEX grant.

\bibliographystyle{alpha}
\bibliography{reference.bib}

\appendix

\section{Approximate Hyperinvertibility}\label{sec:approx-full-spark}
In the following, we prove an approximate version of the hyperinvertibility theorem of Fourier matrices. 
For $q$ prime and $1\le k\le q$, let
\[
   \mathcal{K}_{q,k}:=\{(R,C):R,C\subseteq\Z_q,\ |R|=|C|=k\}.
\]
For a prime $p\equiv1\pmod q$, choose any element $\gamma\in\F_p^\ast$ of order $q$ and define
\begin{equation}\label{eq:Fqkp-def}
   F_{q,k}(p):=
   \Pr_{(R,C)\sim\Unif(\mathcal{K}_{q,k})}
   \bigl(\det(\gamma^{rc})_{r\in R,c\in C}=0\text{ in }\F_p\bigr).
\end{equation}
This value is independent of the chosen order-$q$ element $\gamma$: Replacing $\gamma$ by $\gamma^s$ with $s\in\Z_q^\ast$ replaces $R$ by $sR$, and $R$ is uniform over all $k$-subsets.

\begin{lemma}[Bad Primes for One Submatrix]\label{lem:fixed-minor-bad-primes}
Fix a prime $q$, an integer $1\le k\le q$, and $(R,C)\in\mathcal{K}_{q,k}$.  The number of primes $p\equiv1\pmod q$ for which there exists an element $\gamma\in\F_p^\ast$ of order $q$ such that
\[
   \det(\gamma^{rc})_{r\in R,c\in C}=0\quad\text{in }\F_p
\]
is at most $\frac{(q-1)k}{2}\log_2 k$.
\end{lemma}

\begin{proof}
Define
\[
   D_{R,C}(z):=\det(z^{rc})_{r\in R,c\in C}\in\Z[z],
\]
again using representatives of $rc\in\Z_q$ in $\{0,\ldots,q-1\}$. For $b\in\Z_q^\ast$,
\[
   D_{R,C}(\omega^b)=\det(\omega^{brc})_{r\in R,c\in C}\ne0
\]
by \Cref{thm:tao-chebotarev}.  Also $D_{R,C}(\omega^b)$ is the determinant of a $k\times k$ complex matrix whose entries all have absolute value $1$.  Hadamard's inequality gives
\[
   |D_{R,C}(\omega^b)|\le k^{k/2}.
\]
Apply \Cref{lem:bad-prime-principle} with $B=k^{k/2}$.
\end{proof}
Next, we obtain an average bound over all submatrices.

\begin{lemma}[Averaging over Submatrices]\label{lem:average-minors}
Let $q$ be prime, let $1\le k\le q$, and let $x\ge2$.  For at least half of the primes $p\le x$ with $p\equiv1\pmod q$,
\[
   F_{q,k}(p)
   \le
   \frac{(q-1)k}{\pi(x;q,1)}\log_2 k,
\]
where $\pi(x;q,1):=|\{p\le x:p\text{ prime and }p\equiv1\pmod q\}|$.
\end{lemma}

\begin{proof}
For each admissible prime $p$, fix one order-$q$ element $\gamma_p\in\F_p^\ast$.  By linearity of expectation,
\begin{align*}
   \sum_{\substack{p\le x\\ p\equiv1\, (q)}} F_{q,k}(p)
   &=
   \E_{(R,C)\sim\Unif(\mathcal{K}_{q,k})}
   \sum_{\substack{p\le x\\ p\equiv1\, (q)}}
   \mathbf 1\{\det(\gamma_p^{rc})_{r\in R,c\in C}=0\}.
\end{align*}
For each fixed $(R,C)$, \Cref{lem:fixed-minor-bad-primes} bounds the inner sum by $\frac{(q-1)k}{2}\log_2 k$.  Hence
\[
   \sum_{\substack{p\le x\\ p\equiv1\, (q)}} F_{q,k}(p)
   \le
   \frac{(q-1)k}{2}\log_2 k.
\]
Markov's inequality applied to the nonnegative values $F_{q,k}(p)$ proves the claim.
\end{proof}
We are now ready to state and prove the main theorem of this section.
\begin{theorem}[Approximate Hyperinvertibility]\label{thm:approx-full-spark}
There exist infinitely many primes $q$ such that, for every $1\le k\le q$, at least half of the primes $p\in\mathcal{P}(q)$ satisfy
\[
   \Pr_{(R,C)\sim\Unif(\calK_{q,k})}
   \bigl(\det(\gamma^{rc})_{r\in R,c\in C}\ne0\text{ in }\F_p\bigr)
   \ge
   1-64\frac{k(\log q)^2}{q^2},
\]
where $\gamma\in\F_p^\ast$ is any element of order $q$.
\end{theorem}
\begin{proof}
Let $q$ be any prime satisfying \Cref{eq:many-primes}.  By \Cref{lem:average-minors} with $x=q^4$, at least half of the primes $p\in\mathcal P(q)$ satisfy
\[
   F_{q,k}(p)
   \le
   \frac{(q-1)k}{\pi(q^4;q,1)}\log_2 k \leq32\frac{k(q-1)}{q^3}(\log q)(\log_2 k)
\]
using \Cref{eq:many-primes}. Since $k\le q$ and $\log_2 k\le 2\log q$ for all $k\le q$, this is at most $64k(\log q)^2/q^2$. 
\end{proof}
\end{document}